\documentclass[oneside,11pt]{amsart}
\usepackage[utf8]{inputenc}
\usepackage{amssymb,graphicx,hyperref,microtype}
\usepackage[scale=0.7]{geometry}

\newtheorem{lemma}{Lemma}
\newtheorem{theorem}[lemma]{Theorem}

\newcommand{\cycdist}{\operatorname{dist}_{\mathrm{cyc}}}
\newcommand{\R}{\mathbb{R}}
\newcommand{\Z}{\mathbb{Z}}

\title{On convex spiral equicoverings of masses}

\author[E. Roldán-Pensado]{Edgardo Roldán-Pensado}
\address[E. Roldán-Pensado]{Centro de Ciencias Matemáticas, UNAM Campus Morelia, Morelia, Mexico}
\email{eroldan@matmor.unam.mx}


\begin{document}
	
	\begin{abstract}
		Convex spiral equicoverings were recently introduced by Espinosa-García, Martínez-Sandoval and Roldán-Pensado. They left open the question of whether every planar mass admits a convex $(3k,k+1)$-spiral equicovering. In this paper we give an affirmative answer to this question. To be precise, we prove the following: Given an integer $k \ge 2$, for every planar mass there is a fan consisting of $3k$ equal-mass sectors such that the union of every $k+1$ consecutive sectors is convex.
	\end{abstract}
	
	\maketitle
	
	\section{Introduction}
	
	A planar mass is a Borel probability measure $\mu$ on $\R^2$ such that $\mu(\ell)=0$ for every affine line $\ell$. Mass partition problems ask whether one or more masses can be divided into equal parts using regions with prescribed geometric properties. Classical examples include the ham sandwich theorem, which simultaneously bisects two planar masses with a line, and the theorem of Buck and Buck~\cite{BB49}, which partitions a planar convex body into six regions of equal area using three concurrent lines. For a survey of mass partition problems and their connections with topology, see \cite{RS22}.

    A natural family of partitions is given by fans. A $q$-fan consists of $q$ distinct rays emanating from a common point, and its regions are the sectors between consecutive rays. Simultaneous partitions of measures by fans have been studied by Bárány and Matou\v{s}ek \cite{BM01,BM02}. Among their results, every three planar masses can be simultaneously bisected by a $2$-fan, and every two planar masses can be simultaneously equipartitioned by a $4$-fan.
    
	In \cite{EMR26}, Espinosa-García, Martínez-Sandoval, and Roldán-Pensado introduced equicoverings. Here almost every point is covered with the same multiplicity and the covering regions have equal mass. Spiral equicoverings are obtained from fans by allowing consecutive sectors to overlap, as follows.
	
	Assume $1\le p<q$ and let $r_0,\dots,r_{q-1}$ be distinct rays emanating from a point $O$, listed counterclockwise, with indices taken modulo $q$. Write $V_j$ for the closed sector swept from $r_j$ to $r_{j+1}$ counterclockwise, and let
	\begin{equation}\label{eq:wedges}
		W_j=V_j\cup V_{j+1}\cup\dots\cup V_{j+p-1}.
	\end{equation}
	Every point outside the rays belongs to exactly one sector $V_j$ and to the $p$ wedges $W_{j-p+1},\dots,W_j$. Hence the family $\{W_j\}_{j=0}^{q-1}$ covers almost every point exactly $p$ times (see Figure \ref{fig:equicover}).

    \begin{figure}
        \centering
        \includegraphics{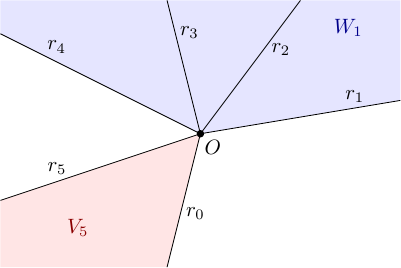}
        \caption{A spiral equicovering with $p=3$ and $q=6$. The sector $V_5$ and the wedge $W_1$ are highligted.}
        \label{fig:equicover}
    \end{figure}
	
	If $\mu(V_j)=1/q$ for every $j$, then $\mu(W_j)=p/q$ for every $j$. Such a family is called a \emph{spiral equicovering}. It is \emph{convex} when all $W_j$ are convex. Every $W_j$ has a counterclockwise angle $\alpha\in(0,2\pi)$ at $O$. Notice that $W_j$ is convex if and only if $\alpha\le\pi$.
	
	For coprime $p$ and $q$, the main theorem in \cite{EMR26} determines the values of $p$ and $q$ for which any mass admits a convex spiral equicovering, leaving open the case when $q=3p-3$ for even $p \ge 4$. Thus, the smallest open case is when $q=9$ and $p=4$ so that the $9$ wedges $W_i$ all have mass of $4/9$ and cover the plane $4$ times. Our main result settles these open cases.
	
	\begin{theorem}\label{thm:main}
		Let $k \ge 2$ be an integer and let $\mu$ be a planar mass. There exist $3k$ distinct counterclockwise rays from a common point $O$ such that every consecutive sector has mass $1/(3k)$ and every union of $k+1$ consecutive sectors has angle at most $\pi$. In other words, $\mu$ admits a convex $(3k,k+1)$-spiral equicovering.
	\end{theorem}

    When $k=2$, every union of three consecutive sectors and its complement both have angle at most $\pi$, so both have angle exactly $\pi$. Thus opposite rays form three concurrent lines, and we recover the six-sector equipartition of a planar mass.

    Our proof follows the configuration-space/test-map approach, with the obstruction expressed through the degree of a boundary map. Similar approaches are used in the fan partition results of \cite{BM01,BM02}, as well as in partitions by hyperplanes with fixed directions \cite{KRS16}. A general introduction to these methods is given in~\cite{Mat03}.
	
	The rest of the paper is devoted to Theorem \ref{thm:main}. We deal with the case when $\mu$ is a planar mass with a strictly positive continuous density, the result for general masses then follows from standard approximation arguments. Throughout the proof we set $n=3k$.
	
	\section{A map on the circle}
	
	Let $\mu$ be any planar mass with a strictly positive continuous density. Given a point $O$ and an angle $\theta$, there is a unique ray $r_0$ emanating from $O$ with direction $(\cos\theta,\sin\theta)$. Starting from this ray, the rays $r_1,\dots,r_{n-1}$ can be constructed in a unique way such that the corresponding sectors $V_0,\dots,V_{n-1}$ have mass $1/n$. Note that the wedges $W_j$, defined as in \eqref{eq:wedges}, have mass $(k+1)/n$.
	Write $\alpha_j\in(0,2\pi)$ for the angle of wedge $W_j$, then these angles depend continuously on the pair $(O,\theta)$.
	
	In order to test the convexity of the $W_j$, we define the function
	\[
	d:\R^2\times S^1\to[0,\pi)^n
	\]
	such that its $j$-th coordinate is given by
	\[d_j(O,\theta)=\max\{0,\alpha_j-\pi\}.\]
	The function $d_i$ is $0$ exactly when its corresponding $W_i$ is convex, therefore, finding a convex spiral equicovering is equivalent to finding a zero of $d$.
	
	Now we construct a simplicial complex $K$ whose vertices are the indices of the $W_i$, namely $\Z_n$. We want the faces of $K$ to contain all possible sets of wedges $W_i$ which are simultaneously non-convex.
	Note that if the cyclic distance $\cycdist(i,j)$ is greater than $k$ for given $i,j\in\Z_n$, then $W_i$ and $W_j$ have disjoint interiors. This implies that they cannot be both non-convex.
	
	Because of this, we define the faces of $K$ to be the subsets $I$ of $\Z_n$ such that $\cycdist(i,j)\le k$ for every $i,j\in I$.
	
	\begin{lemma}\label{lem:faces}
		The maximal faces of $K$ are either an interval of the form
		\[
		\lambda_j=\{j,j+1,\dots,j+k\}\quad\text{for } j\in\Z_n
		\]
		or a triangle of the form
		\[
		\tau_j=\{j,j+k,j+2k\}\quad\text{for } 0\le j<k.
		\]
	\end{lemma}
	\begin{proof}
		Assume we have a non-empty face. Since both types of faces are invariant under cyclic translations, we may assume that $0$ is one of its elements. Every other element must either be in the interval $[1,k]$ or $[2k,3k-1]$. If all of the elements are contained in only one of these two intervals then they are contained in $\lambda_0$ or $\lambda_{2k}$.
		If this is not the case, then let $a$ be the largest element in $[1,k]$ and $b$ be the smallest element in $[2k,3k-1]$. Because $a,b$ are in the same face, we must have $b-a\le k$ or $b-a\ge 2k$.
		
		The first case is only possible when $a=k$ and $b=2k$, furthermore, there cannot be another element in $[1,k]$ or its cyclic distance to $b$ would be too large. Likewise, there cannot be another element in $[2k,3k-1]$. We conclude that this face must be $\tau_0$.
		
		In the second case, the face is contained in $\lambda_b$.

        It is clear from the definition of $K$ that no additional vertex can be added to any $\lambda_j$ or $\tau_j$, hence these are the maximal faces of $K$.
	\end{proof}
	
	Let $L$ be the subcomplex of $K$ generated by the faces of the form $\lambda_j$. Denote by $\lvert L\rvert$ the geometric realization of $L$ so that the vertices of $\lvert L\rvert$ are $e_j$ for $j\in\Z_n$. Identify $S^1$ with the $z\in\mathbb C$ such that $\lvert z\rvert=1$. Then we may define the map
	\[
	Q:\lvert L\rvert \to S^1
	\]
	such that, for every face $I$ of $L$,
	\begin{equation}\label{eq:Q}
		Q\left(\sum_{j\in I} a_j e_j\right) = \frac{\sum_{j\in I} a_j \zeta^j}{\left\lvert\sum_{j\in I} a_j \zeta^j\right\rvert},
	\end{equation}
	where $\zeta = e^{2\pi i/n}$.
	Notice that for every face $I$ there is an arc of length $2\pi/3$ in $S^1$ containing all the $\zeta^j$ with $j\in I$. After rotating this arc so that it is centered at $1$, every $\zeta^j$ in the face has real part at least $1/2$. Thus, the denominator in \eqref{eq:Q} cannot vanish and $Q$ is continuous on $\lvert L\rvert$.
	
	Next, choose a large-enough $R>0$ so that the mass outside of the disk
	\[
	B(R)=\{x\in\R^2:\lvert x\rvert < R\}
	\]
	is smaller than $1/n$.
    
    \begin{lemma}\label{lem:distant}
		If $\lvert O\rvert = R$ then the indices $j$ such that $W_j$ is non-convex determine a non-empty face of $L$.
	\end{lemma}
	\begin{proof}
		Define the opposite closed half-planes
        \begin{align*}
        H^+&=\{x\in\R^2:\langle x,O\rangle\ge R^2\},\\
        H^-&=\{x\in\R^2:\langle x,O\rangle\le R^2\}.
        \end{align*}
        Since $\mu(H^+)<1/n$, at most one ray of the fan is contained in $H^+$. Therefore, after a relabeling, we may assume that the rays $r_1,\dots,r_{n-1}$ are in $H^-$ (see Figure \ref{fig:distant}).
        
        The wedges $W_1,\dots,W_{2k-2}$ are also contained in $H^-$ and are therefore convex. Note that every $\tau_j$ intersects the interval $\{1,\dots,2k-2\}$ so by Lemma \ref{lem:faces} the indices $j$ such that $W_j$ is non-convex are contained in some $\lambda_l$ and therefore determine a face of $L$.
        
        This face is non-empty, since the union of $V_{n_1}$ and $V_1$ is already non-convex, so any wedge $W_j$ containing them is non-convex.
	\end{proof}

    \begin{figure}
        \centering
        \includegraphics{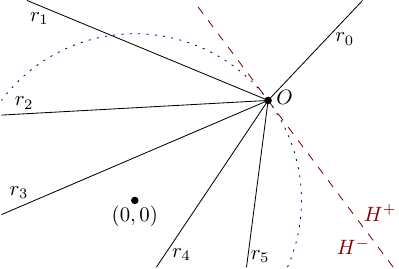}
        \caption{A point $O$ at distance $R$ from the origin and $6$ rays emanating from it, as in the proof from Lemma \ref{lem:distant}.}
        \label{fig:distant}
    \end{figure}
    
    Let
	\[
	M = \overline{B(R)}\times S^1 \quad \text{ and } \quad B = \partial B(R) \times S^1.
	\]
	Lemma \ref{lem:distant} implies that $d$ has at least one positive coordinate at every point in $B$, so we may define a function $f_\mu:B\to\lvert L\rvert$ given by
	\begin{equation}\label{eq:f}
		f_\mu(x)=\sum_{j=0}^{n-1} \frac{d_j(x)}{\sum_{l=0}^{n-1}d_l(x)}e_j.
	\end{equation}

    Assume, for contradiction, that $\mu$ admits no convex $(3k,k+1)$-spiral equicovering. Then $d$ does not vanish anywhere on $M$, so we may define a continuous map of pairs $F_\mu:(M,B)\to(\lvert K\rvert,\lvert L\rvert)$ using the same formula
    \begin{equation}\label{eq:F}
		F_\mu(x)=\sum_{j=0}^{n-1} \frac{d_j(x)}{\sum_{l=0}^{n-1}d_l(x)}e_j.
	\end{equation}
	
	Lastly, for a fixed $\theta_0$, parametrize the boundary of the disk $\overline{B(R)}\times\{\theta_0\}$ counterclockwise by a function
	\begin{align*}
		c:S^1 & \to \partial B(R)\times\{\theta_0\}\\
		e^{i \theta} & \mapsto (R(\cos(\theta),\sin(\theta)),\theta_0)
	\end{align*}
	Then our desired function is the composition
	\[
	S^1 \xrightarrow{c} \partial B(R)\times\{\theta_0\} \xrightarrow{f_\mu} \lvert L\rvert \xrightarrow{Q} S^1.
	\]

    The map $Q\circ f_\mu\circ c$ has a simple informal interpretation. For a given fan, the nonzero coordinates of $f_\mu$ correspond to its non-convex wedges, and their values measure how much their angles exceed $\pi$. The map $Q$ places the indices of these wedges around a circle and takes their normalized weighted average. Thus, $Q\circ f_\mu\circ c$ may be thought of as pointing towards the part of the cyclic ordering where the non-convex wedges occur.

    The following two sections compute the degree of this map in different ways. Assuming that
    In Section~\ref{sec:deg1}, the extension $F_\mu$, together with the cyclic symmetry of the construction, implies that the degree of $Q\circ f_\mu\circ c$ is divisible by $k$.

    On the other hand, we expect the non-convex wedges of a fan whose center is far from most of the mass to occur roughly in the direction pointing away from the mass. When moving the center around the origin, this direction also travels once around it. In Section \ref{sec:deg2} we make this precise by deforming the mass to one with radial symmetry. The resulting computation shows that the of degree of $Q\circ f_\mu\circ c$ is congruent to $1$ modulo $3k$.
	
	\section{The degree computed via a cyclic shift}\label{sec:deg1}
	
	Here we list some of the topological properties of the maps and spaces defined in the previous section in order to compute the degree of our function. For the sake of contradiction, we assume that Theorem \ref{thm:main} is false. This allows us to take advantage of the existence of $F_\mu$.
	
	Let's start with the space $\R^2\times S^1$. There is a natural order-$n$ homeomorphism $T$ on this space that fixes the point $O$ and increases the initial angle $\theta$ so that the initial ray $r_0$ becomes the next ray $r_1$ of the fan. This acts on $d$ by
	\begin{equation}\label{eq:T}
		d_j\circ T(O,\theta)=d_{j+1}(O,\theta).    
	\end{equation}
	Note that $T$ is isotopic to the identity, since the angle from $r_0$ to $r_1$ can be varied continuously in $(0,2\pi)$ while leaving $O$ fixed.
	
	Now consider the simplicial complex $K$ and its subcomplex $L$. By Lemma \ref{lem:faces}, the only simplices of $K$ not contained in $L$ are the triangles $\tau_j$, while every proper face of each $\tau_j$ belongs to $L$. Therefore, by giving $\tau_j=[j,j+k,j+2k]$ the indicated orientation,
	\[
	C_3(K,L;\mathbb Z) = 0, \quad
	C_2(K,L;\mathbb Z) = \bigoplus_{j=0}^{k-1}\mathbb Z[\tau_j], \quad
	C_1(K,L;\mathbb Z)=0.
	\]
	Thus both boundary homomorphisms adjacent to $C_2(K,L;\mathbb Z)$ vanish, and consequently
	\[
	H_2(K,L;\mathbb Z) = \bigoplus_{j=0}^{k-1}\mathbb Z[\tau_j].
	\]
	
	Moreover, the cyclic shift
	\begin{align*}
		P:K & \to K\\
		i & \mapsto i+1
	\end{align*}
	preserves $L$ and satisfies $P_*[\tau_j]=[\tau_{j+1}]$, where the indices are taken modulo $k$. Note also that composing $Q$ with $P$ gives
	\begin{equation}\label{eq:QP}
	    Q(Px) = \zeta Q(x),
	\end{equation}
	where $\zeta = e^{2\pi i/n}$.
	
	The boundary of $\tau_j$ is the cycle traversing the vertices $j,j+k,j+2k$ in this order, which is composed of three edges. Under $Q$, each edge is sent to the corresponding counterclockwise arc of angular length $2\pi/3$, so these wind once around the $S^1$. Therefore, for the composition
	\[
	H_2(K,L;\Z) \xrightarrow{\delta} H_1(L;\Z) \xrightarrow{Q_*} H_1(S^1;\Z) \cong \Z,
	\]
	where we identify $H_1(S^1;\mathbb Z)$ with $\mathbb Z$ using the counterclockwise orientation, we have
	\begin{equation}\label{eq:Q*delta}
		\sum_{j=0}^{k-1} a_j [\tau_j]\mapsto\sum_{j= 0}^{k-1} a_j.
	\end{equation}
	
	Note that applying $T$ from \eqref{eq:T} in $M$ corresponds to a shift in $\lvert K\rvert$, so
	\begin{equation}\label{eq:FT}
		F_\mu\circ T=P^{-1}\circ F_\mu.
	\end{equation}
	
	Let $m\in H_2(M,B;\Z)$ be the class of the oriented disk $\overline{B(R)}\times\{\theta_0\}$ for some initial direction $\theta_0$. Note that $T$ is isotopic to the identity as a map of pairs, so $T_*$ fixes $m$. Thus
	\[
	P_*^{-1}(F_\mu)_*m = (F_\mu)_*T_*m = (F_\mu)_*m.
	\]
	Because $P^{-1}_*$ cyclically permutes the basis, we have 
	\[
	(F_\mu)_*(m) = a\sum_{j=0}^{k-1} [\tau_j]
	\]
	for some integer $a$.
	Let
	\[
	\delta_M: H_2(M,B;\mathbb Z)\to H_1(B;\mathbb Z)
	\]
	and
	\[
	\delta_K: H_2(K,L;\mathbb Z)\to H_1(L;\mathbb Z)
	\]
	denote the corresponding connecting homomorphisms. Then, by naturality of the connecting homomorphisms,
	\[
	(F_\mu|_B)_*(\delta_Mm)
	=\delta_K((F_\mu)_*m).
	\]
	Since $\delta_Mm=c_*[S^1]$, we obtain
	\begin{align*}
		\deg(Q\circ F_\mu\circ c)
		&=Q_*(F_\mu|_B)_*(\delta_Mm)\\
		&=Q_*\delta_K((F_\mu)_*m)\\
		&=ka.
	\end{align*}
	
	\section{The degree computed via a measure with radial symmetry}\label{sec:deg2}
	
	We now compute the degree of the same function by deforming $\mu$ into a measure with radial symmetry. Let $\rho$ be a planar mass with a strictly positive continuous density, invariant under rotations about the origin, such that
	\[
	\rho(\R^2\setminus B(R))<1/n.
	\]
	For example, we may take a Gaussian measure concentrated around the origin. For $0\le s\le 1$, let
	\[
	\mu_s = (1-s)\mu+s\rho.
	\]
	Every $\mu_s$ has a strictly positive continuous density and satisfies
	\[
	\mu_s(\R^2\setminus B(R))<1/n.
	\]
	Then, using the formula in \eqref{eq:f}, we have a continuous family of maps
	\[
	f_{\mu_s}|_B:B\to |L|.
	\]
	Here continuity follows from the continuous dependence of the equal-mass rays on the center, the initial direction, and $s$. It follows that
	\[
	\deg(Q\circ F_\mu\circ c) = \deg(Q\circ f_\rho\circ c).
	\]
	
	To compute the latter degree, we fix $O_0=(R,0)$ and vary the initial direction. Define
    \begin{align*}
        h:S^1 & \to S^1\\
        e^{i\theta} & \mapsto Q(f_\rho(O_0,\theta))
    \end{align*}
	We use the angular mass around $O_0$ to parametrize the initial direction. For $0\le\theta<2\pi$, let $A(\theta)$ be the $\rho$-mass of the sector swept counterclockwise from direction $0$ to direction $\theta$, and extend $A$ to $\R$ by
	\[
	A(\theta+2\pi) = A(\theta)+1.
	\]
	This is a continuous strictly increasing function, so it has an inverse and $A^{-1}(t+1)=A^{-1}(t)+2\pi$. In this parametrization, advancing to the next ray adds $1/n$ to $t$.
	
	Define
	\[
	g(e^{2\pi it})=h(e^{iA^{-1}(t)}).
	\]
	The change of parametrization preserves orientation, so $\deg g=\deg h$. By \eqref{eq:FT} and \eqref{eq:QP}, we have
	\[
	g(e^{2\pi i(t+1/n)}) = e^{-2\pi i/n}g(e^{2\pi it}).
	\]
	Choose a continuous lift $u:\R\to\R$ such that
	\[
	g(e^{2\pi it})=e^{2\pi iu(t)}.
	\]
	Then
	\[
	u(t+1/n)-u(t)+1/n
	\]
	is an integer that depends continuously on $t$, and is therefore equal to some constant $b\in\Z$. Applying this identity $n$ times gives
	\[
	\deg h=\deg g=u(t+1)-u(t)=nb-1.
	\]

    \begin{figure}
        \centering
        \includegraphics{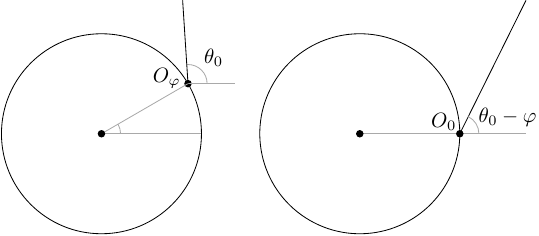}
        \caption{The way in which $O_\varphi$ and $\theta_0$ change after a rotation of $-\varphi$.}
        \label{fig:rotation}
    \end{figure}

    Finally, fix $\varphi$ and consider the fan with common vertex $O_\varphi=R(\cos\varphi,\sin\varphi)$ and initial direction $\theta_0$. Rotate the plane clockwise through angle $\varphi$ about the origin. This sends $O_\varphi$ to $O_0=(R,0)$, changes the initial direction to $\theta_0-\varphi$ but leaves $\rho$ invariant (see Figure \ref{fig:rotation}). Therefore the rotated fan is the equal-mass fan associated with $(O_0,\theta_0-\varphi)$. This rotation preserves the labels of the rays and the angles of the wedges, so
    \[
    f_\rho(O_\varphi,\theta_0)
     =f_\rho(O_0,\theta_0-\varphi).
    \]
    Therefore,
	\[
	(Q\circ f_\rho\circ c)(e^{i\varphi}) = h(e^{i(\theta_0-\varphi)}).
	\]
	Since $\varphi\mapsto\theta_0-\varphi$ reverses orientation,
	\[
	\deg(Q\circ F_\mu\circ c) = \deg(Q\circ f_\rho\circ c) = -\deg h = 1-nb.
	\]
	The previous section shows that this degree is divisible by $k$.
	Since $n=3k$, we obtain
	\[
	ka = 1-3kb,
	\]
	which is impossible. This proves Theorem~\ref{thm:main}.
	
	\section{Acknowledgments}
	This work was supported by UNAM-PAPIIT project IN114726. The author used OpenAI's ChatGPT in developing the proof and in drafting and revising the manuscript, especially in formalizing the topological arguments. The author verified every claim in the final proof and takes responsibility for the mathematical arguments.

\end{document}